\documentclass[11pt]{amsart}

\usepackage[margin=1.12in]{geometry}
\usepackage{microtype}
\usepackage{mathtools,amssymb,amsfonts}
\usepackage{enumitem}
\usepackage{xcolor}
\usepackage[colorlinks=true,linkcolor=blue!55!black,citecolor=green!35!black,urlcolor=blue!60!black]{hyperref}
\usepackage[capitalise,noabbrev]{cleveref}

\newtheorem{theorem}{Theorem}
\newtheorem{proposition}[theorem]{Proposition}
\newtheorem{lemma}[theorem]{Lemma}

\newtheorem{markloflemma}{Lemma}[section]
\crefname{markloflemma}{Lemma}{Lemmas}
\Crefname{markloflemma}{Lemma}{Lemmas}
\theoremstyle{definition}

\newtheorem{marklofdefinition}[markloflemma]{Definition}
\crefname{marklofdefinition}{Definition}{Definitions}
\Crefname{marklofdefinition}{Definition}{Definitions}
\theoremstyle{remark}
\newtheorem{remark}[theorem]{Remark}

\newcommand{\Z}{\mathbb Z}
\newcommand{\Q}{\mathbb Q}
\newcommand{\R}{\mathbb R}
\newcommand{\C}{\mathbb C}
\newcommand{\cC}{\mathcal C}
\newcommand{\one}{\mathbf 1}
\newcommand{\e}{\mathrm e}
\DeclareMathOperator{\vol}{vol}
\title[Berry--Tabor statistics for the Sutherland model]
{Poissonian pair correlation for the three-particle Sutherland model}
\author{Christopher Lutsko}
\date{}

\begin{document}

\begin{abstract}
We prove a Berry--Tabor theorem for the repulsive three-particle trigonometric
Sutherland Hamiltonian.  For fixed total momentum, its relative spectrum is an
affine translate of the positive-definite $A_2$-form.  For a Diophantine
coupling, the desymmetrized spectrum has Poissonian pair correlation.  The proof
reduces the problem to a finite-congruence and Weyl-sector version of Marklof's
theorem on inhomogeneous quadratic forms. 
\end{abstract}

\maketitle

\section{Introduction}

In 1977 Berry and Tabor \cite{BerryTabor1977} conjectured that a generic integrable quantum system has
Poissonian local spectral statistics once its exact symmetries have been removed.
We prove this at the level of pair correlation for a
natural interacting model: the three-particle trigonometric Sutherland
Hamiltonian with interaction coupling $g>1$,
\begin{equation}\label{eq:hamiltonian-intro}
 H_g=-\frac12\sum_{j=1}^3\frac{\partial^2}{\partial x_j^2}
 +\frac{g(g-1)}4\sum_{1\leq i<j\leq3}
 \frac1{\sin^2((x_i-x_j)/2)}.
\end{equation}
This describes three particles on a circle with repulsive inverse-square chord
interaction.  The model is quantum integrable and exactly solvable: its
quasimomenta form an affine translate of the dominant weight lattice, and its
eigenfunctions can be expressed in terms of Jack polynomials
\cite{Sutherland1971,Sutherland1972,OlshanetskyPerelomov1983,
BakerForrester1997,Hallnas2024}.  An explicit treatment of the three-particle
system, including formulas for the corresponding Jack-polynomial
eigenfunctions, is given by Mangazeev \cite{Mangazeev2007}.

Fixing the total momentum $K\in\Z$ and subtracting the center-of-mass energy
$K^2/6$, the relative spectrum is
\begin{equation}\label{eq:intro-form}
 E_{p,q}:=\frac13\bigl((p+g)^2+(p+g)(q+g)+(q+g)^2\bigr),
 \qquad p,q\in\Z_{\geq0},\quad p+2q\equiv K\pmod3.
\end{equation}
This is the three-particle specialization of the standard Sutherland
quasimomentum formula; see \cite[Section~4.2]{Hallnas2024}. 
Thus the interaction parameter is a shift of the positive-definite $A_2$-form.
The only remaining issues are the congruence condition, the Weyl symmetries (taking $p,q\ge0$ restricts to one chamber), and the
reflection symmetry when $K\equiv0\pmod3$.

We remove this last symmetry at the outset.  Put
\[
 \mathcal P_K:=
 \begin{cases}
 \{(p,q)\in\Z_{\geq0}^2:p+2q\equiv K\pmod3\},
       &K\not\equiv0\pmod3,\\
 \{(p,q)\in\Z_{\geq0}^2:p+2q\equiv0\pmod3,\ p\geq q\},
       &K\equiv0\pmod3,
 \end{cases}
\]
and let $\mathcal E_{K,g}:=\{E_{p,q}:(p,q)\in\mathcal P_K\}$.  Its mean
density is
\[
 \rho_K:=\begin{cases}
 \displaystyle\frac{\pi}{3\sqrt3},&K\not\equiv0\pmod3,\\[6pt]
 \displaystyle\frac{\pi}{6\sqrt3},&K\equiv0\pmod3.
 \end{cases}
\]
Write
\[
 N_{K,g}(X):=\#\{E\in\mathcal E_{K,g}:E\leq X\},
\]
and, for $f\in C_c^\infty(\R)$, define the pair correlation
\[
 R_{K,g}(X,f):=\frac1{N_{K,g}(X)}
 \sum_{\substack{E,E'\in\mathcal E_{K,g}\\E,E'\leq X}}^\ast
 f\bigl(\rho_K(E-E')\bigr),
\]
where the $\ast$ denotes that $E\neq E'$

For every irrational $g>1$ and every $K\in\Z$, the following Weyl law
applies:
\begin{equation}\label{eq:intro-weyl}
 N_{K,g}(X)=\rho_KX+O_{g,K}(X^{1/2}).
\end{equation}
Call $g\in\R$ \emph{Diophantine of finite type} if there are $c>0$ and
$\kappa<\infty$ such that
\begin{equation}\label{eq:diophantine}
 \left|g-\frac ab\right|\geq c b^{-\kappa}
 \qquad(a\in\Z,\ b\in\Z_{\geq1}).
\end{equation}
This includes every algebraic irrational and almost every real number. 
Some Diophantine restriction is necessary in general: without
it pair correlation can fail along subsequences
\cite[Theorem~1.13]{Marklof2003}; see also the higher-dimensional correction
\cite{MarklofCorrection2003}.

\begin{theorem}
\label{thm:main}
Let $g>1$ be Diophantine of finite type, and let $K\in\Z$.
Then
\[
 \lim_{X\to\infty}R_{K,g}(X,f)=\int_\R f(x)\,dx
\]
for every $f\in C_c^\infty(\R)$.  In particular, $\mathcal E_{K,g}$ has
Poissonian pair correlation.
\end{theorem}

\begin{remark}
For the $N$-particle model, after fixing the total momentum, the relative
energies come from a shifted positive-definite $(N-1)$-ary quadratic form.
For $N\geq4$ this creates additional collisions.
\end{remark}

There is a substantial physics literature on many-particle analogues of the
Berry--Tabor conjecture.  Poisson statistics have been observed numerically in
integrable Heisenberg, $t$--$J$, and Hubbard models
\cite{PoilblancEtAl1993}.  Noninteracting many-particle spectra were studied in
\cite{MunozEtAl2006}.  For interacting bosons, the level statistics of the
Lieb--Liniger model were investigated analytically and numerically in
\cite{SekkouriIzrailevBorgonovi2021}, and an $N$-particle Berry--Tabor trace
formula was derived in \cite{UrbinaKellyRichter2023}.  That said, these works do not give
a rigorous Poissonian pair-correlation limit for a deterministic interacting
Hamiltonian. To our knowledge, \cref{thm:main} is the first rigorous proof of Poissonian pair correlation for a deterministic quantum Hamiltonian with multiple interacting particles.

Rigorous work on Berry--Tabor statistics
has largely concerned free particles on flat tori and boxes.  Sarnak proved
pair correlation for almost every positive binary quadratic form
\cite{Sarnak1997}; Eskin, Margulis and Mozes obtained fixed-form results under
explicit Diophantine hypotheses \cite{EskinMargulisMozes2005}.  Marklof proved
Poissonian pair correlation for shifted lattice norms using theta functions and
Ratner theory \cite{Marklof2002,Marklof2003}. Effective refinements were developed by
Str\"ombergsson--Vishe and Lindenstrauss--Mohammadi--Wang
\cite{StrombergssonVishe2020,LindenstraussMohammadiWang2023}.  In dimension
three, Kim, Marklof and Welsh recently proved the pair-correlation conjecture
for a Diophantine family of boxes \cite{KimMarklofWelsh2026}. 

The closest input to our argument is Marklof's inhomogeneous theorem,
particularly its symmetry-reduced cases \cite[Appendix A]{Marklof2003}.  In what follows, we pass from the square lattice to a rational $A_2$-form, a finite-index
coset, and a sharp Weyl sector.  The first two are handled by the finite Weil
representation, and the sector by smooth approximation and nonescape of mass.
The weighted theta mean-value statement used below is proved in
\cite[Proposition~8.5]{GriffinMarklof2019}.

\section{Exact spectrum and center-of-mass reduction}\label{sec:spectrum}

We work in the bosonic Hilbert space for three particles on the circle. Fix labels
\begin{equation}\label{eq:lambda}
 \lambda=(\lambda_1,\lambda_2,\lambda_3),\qquad
 \lambda_1\geq\lambda_2\geq\lambda_3,
\end{equation}
then the quasimomenta are given by
\[
 \kappa_i=\lambda_i+\frac g2(4-2i),\qquad
 (\kappa_1,\kappa_2,\kappa_3)
 =(\lambda_1+g,\lambda_2,\lambda_3-g).
\]
The energy and total momentum are
\[
 E_\lambda=\frac12\sum_{i=1}^3\kappa_i^2,
 \qquad K_\lambda=\sum_{i=1}^3\kappa_i
 =\lambda_1+\lambda_2+\lambda_3.
\]
One often indexes Jack polynomials by nonnegative partitions
\cite[Section 2]{Mangazeev2007}.

For fixed $K$ and $\lambda$, we write
\[
 p=\lambda_1-\lambda_2,\qquad q=\lambda_2-\lambda_3,
\]
so $p,q\geq0$.  The total-momentum condition is
\begin{equation}\label{eq:congruence}
 K=3k+p+2q,\qquad p+2q\equiv K\pmod3,
\end{equation}
which gives
\[
 \lambda=(k+p+q,k+q,k).
\]
The following lemma is the standard quasimomentum formula for the Sutherland
model \cite[Section~2]{Mangazeev2007}.

\begin{lemma}\label{lem:relative-spectrum}
For fixed $K$, the energy after subtracting $K^2/6$ is
\begin{equation}\label{eq:relative-energy}
 E^{\mathrm{rel}}_{p,q}
 =E_\lambda-\frac{K^2}{6}
 =\frac13Q_g(p,q),
\end{equation}
where
\begin{align*}
 Q_g(p,q)
 &=(p+g)^2+(p+g)(q+g)+(q+g)^2\\
 &=p^2+pq+q^2+3g(p+q)+3g^2.
\end{align*}
The allowed labels are exactly
\begin{equation}\label{eq:labelset}
 \mathcal L_K=\{(p,q)\in\Z_{\geq0}^2:p+2q\equiv K\pmod3\}.
\end{equation}
\end{lemma}

\begin{proof}
For any $(a_1,a_2,a_3)\in\R^3$ one has
\[
 \sum_{i=1}^3a_i^2-\frac13\Bigl(\sum_{i=1}^3a_i\Bigr)^2
 =\frac13\sum_{1\leq i<j\leq3}(a_i-a_j)^2.
\]
Apply this identity with $a_i=\kappa_i$ and use
$E_\lambda=\frac12\sum_i\kappa_i^2$ and $\sum_i\kappa_i=K$.  We obtain
\[
 E_\lambda-\frac{K^2}{6}
 =\frac16\bigl((\kappa_1-\kappa_2)^2
 +(\kappa_1-\kappa_3)^2+(\kappa_2-\kappa_3)^2\bigr).
\]
The definitions of $p$ and $q$ give
\[
 \kappa_1-\kappa_2=p+g,\qquad
 \kappa_2-\kappa_3=q+g,\qquad
 \kappa_1-\kappa_3=p+q+2g.
\]
Consequently
\begin{align*}
 6\Bigl(E_\lambda-\frac{K^2}{6}\Bigr)
 &=(p+g)^2+(q+g)^2+(p+q+2g)^2\\
 &=2\{p^2+pq+q^2+3g(p+q)+3g^2\}=2Q_g(p,q),
\end{align*}
which proves \eqref{eq:relative-energy}.

If $\lambda$ has momentum $K$, then \eqref{eq:congruence} implies
$p+2q\equiv K\pmod3$.  Conversely, given $p,q\geq0$ satisfying this
congruence, the integer
\[
 k=\frac{K-p-2q}{3}
\]
defines a weakly decreasing integral triple
$\lambda=(k+p+q,k+q,k)$ of momentum $K$.  These two constructions are
inverse to one another, proving \eqref{eq:labelset}.
\end{proof}

Next, the change of variables
\begin{equation}\label{eq:uv}
 u=p+q,\qquad v=p-q
\end{equation}
puts the form in the especially useful shape
\[
 4Q_g(p,q)=3(u+2g)^2+v^2.
\]
The image of $\mathcal L_K$ is the finite-index lattice coset
\begin{equation}\label{eq:uv-lattice}
 u\equiv v\pmod2,\qquad v\equiv K\pmod3,\qquad u\geq|v|.
\end{equation}
The first two conditions are a parity sublattice and a residue class; the last
condition is the closed $A_2$ Weyl chamber.

\section{Arithmetic: multiplicities and mean density}\label{sec:arithmetic}

Let $Q_0(p,q)=p^2+pq+q^2$.  The next lemma states that irrationality of $g$ and working in the positive Weyl chamber, the only remaining symmetry to be removed is the reflection when $K\equiv 0$.

\begin{lemma}\label{lem:collision}
If $g$ is irrational and
$Q_g(p,q)=Q_g(r,s)$ for nonnegative integral labels, then
\[
 (r,s)=(p,q)\quad\hbox{or}\quad(r,s)=(q,p).
\]
If both labels lie in $\mathcal L_K$ and the two labels are distinct, the
second possibility can occur only when $K\equiv0\pmod3$.
\end{lemma}

\begin{proof}
The equality is
\[
 Q_0(p,q)-Q_0(r,s)+3g(p+q-r-s)=0.
\]
Both coefficients of $1$ and $g$ are integers, so irrationality gives
$p+q=r+s$ and $Q_0(p,q)=Q_0(r,s)$.  Since
$Q_0(p,q)=(p+q)^2-pq$, we also have $pq=rs$ which forces $(r,s)=(p,q)$ or $(q,p)$.

Suppose that the distinct labels $(p,q)$ and $(q,p)$ both belong to
$\mathcal L_K$.  Since $p+2q\equiv K\pmod3$, the reflected congruence is
$q+2p\equiv-K\pmod3$.  Thus the reflected point lies in $\mathcal L_K$
only if $K\equiv-K\pmod3$, which is equivalent to $K\equiv0\pmod3$.
\end{proof}

We remove this remaining symmetry by restricting to $\mathcal{P}_K$. The following Weyl law is an easy consequence of lattice point equidistribution.

\begin{lemma}\label{lem:weyl}
For every irrational $g>1$ and $K\in\Z$,
\[
 N_{K,g}(X)=\rho_KX+O_{g,K}(X^{1/2}).
\]
\end{lemma}

\begin{proof}
Let
\[
 B=\begin{pmatrix}1&1/2\\1/2&1\end{pmatrix},
 \qquad Q_0(x,y)=(x,y)B(x,y)^t.
\]
Since $\det B=3/4$, the change $w=B^{1/2}(x,y)^t$ gives
\[
 \vol\{(x,y)\in\R^2:Q_0(x,y)\leq Y\}
 =\frac{\pi Y}{\sqrt{\det B}}=\frac{2\pi Y}{\sqrt3}.
\]
The $B$-inner product of $(1,0)$ and $(0,1)$ is $1/2$, and both vectors
have $B$-norm one.  Thus the angle between the two boundary rays of the
positive quadrant is $\arccos(1/2)=\pi/3$ in the $B$-metric.  The quadrant
therefore occupies one sixth of the ellipse, so
\begin{equation}\label{eq:sector-area}
 \vol\{(x,y)\in\R_{\geq0}^2:Q_0(x,y)\leq Y\}
 =\frac{\pi Y}{3\sqrt3}.
\end{equation}

In fact
\[
 Q_g(p,q)=Q_0(p+g,q+g),
\]
so the relevant ellipse is translated by the fixed vector $(g,g)$.  Hence
\[
 \Omega_{g,Y}:=\{z\in\R_{\geq0}^2:Q_g(z)\leq Y\}
\]
has boundary consisting of a bounded number of $C^1$ arcs of total length
$O_g(Y^{1/2})$.  The planar Lipschitz principle gives, for every fixed
full-rank sublattice $L\subset\Z^2$ and coset $c+L$,
\[
 \#((c+L)\cap\Omega_{g,Y})
 =\frac{\vol(\Omega_{g,Y})}{[\Z^2:L]}+O_{g,L}(Y^{1/2}).
\]
Translating the ellipse by the fixed vector $(g,g)$ changes its intersection
with the quadrant only in a strip of bounded width along a boundary of length
$O_g(Y^{1/2})$.  It follows from \eqref{eq:sector-area} that
\[
 \vol(\Omega_{g,Y})=\frac{\pi Y}{3\sqrt3}+O_g(Y^{1/2}).
\]

The congruence $p+2q\equiv K\pmod3$ defines a coset of the sublattice
\[
 L_0=\{(p,q)\in\Z^2:p+2q\equiv0\pmod3\},
 \qquad [\Z^2:L_0]=3.
\]
Taking $Y=3X$ because $E^{\mathrm{rel}}=Q_g/3$, we therefore obtain
\[
 \#\{(p,q)\in\mathcal L_K:Q_g(p,q)\leq3X\}
 =\frac13\frac{\pi(3X)}{3\sqrt3}+O_{g,K}(X^{1/2})
 =\frac{\pi X}{3\sqrt3}+O_{g,K}(X^{1/2}).
\]
When $K\equiv0\pmod3$, both the coset and $Q_g$ are invariant under
$(p,q)\mapsto(q,p)$.  The Weyl chamber splits into two open half-chambers
with the same cardinality.  The wall $p=q$ contains $O_g(X^{1/2})$ labels
below energy $X$.  Thus restriction to $p\geq q$ halves the main term without
changing the order of the error.  This gives the stated value of $\rho_K$.
\end{proof}

\section{Theta sums and Marklof's theorem}\label{sec:marklof-prelim}

\subsection*{The Jacobi quotient and the affine horocycle}

Write $\e(t)=e^{2\pi it}$.  For $k\geq1$ let
\[
 G^k=\operatorname{SL}(2,\R)\ltimes\R^{2k}.
\]
We write an element of $\R^{2k}$ as
$\xi={}^t(\boldsymbol x,\boldsymbol y)$ with
$\boldsymbol x,\boldsymbol y\in\R^k$, and let
$M\in\operatorname{SL}(2,\R)$ act on each of the $k$ column vectors
${}^t(x_j,y_j)$.  Thus the group law is
\[
 (M;\xi)(M';\xi')=(MM';\xi+M\xi').
\]
Put
\[
 n(u)=\begin{pmatrix}1&u\\0&1\end{pmatrix},\qquad
 a(v)=\begin{pmatrix}v^{1/2}&0\\0&v^{-1/2}\end{pmatrix}
 \quad(v>0),\qquad
 k(\phi)=\begin{pmatrix}\cos\phi&-\sin\phi\\
                         \sin\phi& \cos\phi\end{pmatrix}.
\]
If $f\in\mathcal S(\R^k)$ and
$g=(n(u)a(v)k(\phi);{}^t(\boldsymbol x,\boldsymbol y))$, define
\begin{equation}\label{eq:marklof-theta}
 \Theta_f(g)=v^{k/4}\sum_{m\in\Z^k}
 f_\phi((m-\boldsymbol y)v^{1/2})
 \e\!\left(\frac u2\lVert m-\boldsymbol y\rVert^2
                     +m\cdot\boldsymbol x\right).
\end{equation}
Here $f_\phi=\widetilde R(i,\phi)f$, where
$\widetilde R$ is the metaplectic representation; only the facts
$f_0=f$ and $\lVert f_\phi\rVert_2=\lVert f\rVert_2$ will be used.
Define $\Gamma_\theta^k$ to be the common stabilizer
\[
 \left\{\gamma\in\operatorname{SL}(2,\Z)\ltimes(\tfrac12\Z)^{2k}:
 (\Theta_f\overline{\Theta_{f'}})(\gamma g)
 =(\Theta_f\overline{\Theta_{f'}})(g)
 \text{ for all }f,f'\in\mathcal S(\R^k),\ g\in G^k\right\}.
\]
The theta transformation formulas obtained from Poisson summation show that
$\Gamma_\theta^k$ has finite index in
$\operatorname{SL}(2,\Z)\ltimes(\tfrac12\Z)^{2k}$
\cite[Lemmas~4.11--4.12]{Marklof2003}.  Consequently each product is a
function on the finite-volume quotient $\Gamma_\theta^k\backslash G^k$.

The rational dependence relevant here has a particularly concrete
description.  We call a subgroup of $G^2$ \emph{arithmetic} if it is
commensurable with
$\operatorname{SL}(2,\Z)\ltimes\Z^4$; replacing $\Z^4$ by a rational
dilate gives the same class.  Inside $G^2$ set
\[
 H=\{(M;{}^t(x,0,y,0)):M\in\operatorname{SL}(2,\R),\ x,y\in\R\}.
\]
The map $(M;{}^t(x,0,y,0))\mapsto(M;{}^t(x,y))$ identifies $H$ with $G^1$.
For an arithmetic subgroup $\Gamma<G^2$, put $\Gamma_H=\Gamma\cap H$ and
\[
 Y=\Gamma_H\backslash H.
\]
The group $\Gamma_H$ is a lattice in $H$, so $Y$ has finite Haar measure.  In
the coordinates $(u,v,\phi,x,y)$ above this measure is a constant multiple of
\[
 \frac{du\,dv\,d\phi\,dx\,dy}{v^2};
\]
we write $\mu_Y$ for its normalization to total mass one.
For $\alpha\in\R$ let
\[
 g_\alpha=(I;{}^t(0,0,-\alpha,0))\in H.
\]
The point represented in Iwasawa coordinates by
\[
 (u+iv,0;{}^t(0,0,-\alpha,0))
\]
is precisely $\Gamma_Hg_\alpha n(u)a(v)$.  Thus the limit $v\to0$ is an expanding affine
horocycle in $Y$.

\begin{lemma}\label{lem:theta-identity}
Let $M$ be a discrete set and let $\lambda_m>0$ for $m\in M$.  For
$\psi_1,\psi_2\in C_c^\infty(\R_{>0})$ define
\[
 S_{j,X}(u)=X^{-1/2}\sum_{m\in M}
 \psi_j\!\left(\frac{\lambda_m}{X}\right)
 \e\!\left(\frac u2\lambda_m\right).
\]
For $h\in C_c(\R)$ set
\[
 \widehat h(s)=\int_\R h(u)\e(us/2)\,du.
\]
Then
\begin{equation}\label{eq:theta-pair-identity}
 \int_\R S_{1,X}(u)\overline{S_{2,X}(u)}h(u)\,du
 =\frac1X\sum_{m,n\in M}
 \psi_1\!\left(\frac{\lambda_m}{X}\right)
 \psi_2\!\left(\frac{\lambda_n}{X}\right)
 \widehat h(\lambda_m-\lambda_n).
\end{equation}
If $\lambda_m$ is a positive quadratic polynomial on a lattice coset, the
left side of \eqref{eq:theta-pair-identity}, with $v=X^{-1}$, is a finite
linear combination of integrals of products of the form
$\Theta_f\overline{\Theta_{f'}}(g_\alpha n(u)a(v))$.
\end{lemma}

\begin{proof}
Because the $\psi_j$ have compact support and the quadratic polynomial is
proper, both sums are finite.  We may therefore multiply them and integrate
term by term.  The summand indexed by $(m,n)$ contains
\[
 \int_\R h(u)\e\!\left(\frac u2(\lambda_m-\lambda_n)\right)du
 =\widehat h(\lambda_m-\lambda_n),
\]
which proves \eqref{eq:theta-pair-identity}.  If
$\lambda_m=(m-\boldsymbol y)^tA(m-\boldsymbol y)$, take a Schwartz function
whose value at $w=(m-\boldsymbol y)X^{-1/2}$ is
$\psi_j(w^tAw)$.  Formula \eqref{eq:marklof-theta} then gives the asserted
theta representation.  A congruence coset and rational $A$ require finitely
many component theta functions; these are constructed in
\cref{lem:rational-theta} below.
\end{proof}

\subsection*{Height, domination, and nonescape}

Let $\Gamma_\infty=\{n(m):m\in\Z\}$.  If $\tau=u+iv$ and
$\gamma=\left(\begin{smallmatrix}a&b\\c&d\end{smallmatrix}\right)$, put
\[
 v_\gamma=\operatorname{Im}(\gamma\tau)
 =\frac{v}{|c\tau+d|^2},\qquad
 {}^t(\boldsymbol x_\gamma,\boldsymbol y_\gamma)
 =\gamma{}^t(\boldsymbol x,\boldsymbol y).
\]
The modular height is the well-defined function
\[
 \mathcal H(\Gamma g)=
 \sup_{\gamma\in\operatorname{SL}(2,\Z)}v_\gamma.
\]
For a nonnegative, even, rapidly decreasing function $f_0\in C(\R)$ and
$R>1$, define Marklof's cusp majorant
\begin{equation}\label{eq:cusp-majorant}
 F_R(g)=
 \sum_{\gamma\in\Gamma_\infty\backslash\operatorname{SL}(2,\Z)}
 \sum_{m\in\Z}
 f_0((y_{1,\gamma}+m)v_\gamma^{1/2})v_\gamma
 \one_{[R,\infty)}(v_\gamma).
\end{equation}
The sum is invariant under
$\operatorname{SL}(2,\Z)\ltimes\Z^{2k}$.  A continuous function $F$ on an
arithmetic Jacobi quotient is said to be \emph{dominated by $F_R$} if there
is an $L>1$ such that, for all sufficiently large $R$,
\begin{equation}\label{eq:dominated}
 |F(g)|\one_{\{\mathcal H(g)\geq R\}}\leq L+F_R(g)
 \qquad(g\in G^k).
\end{equation}
This is Marklof's precise definition of moderate growth in the cusp
\cite[Section~7.2]{Marklof2003}.  The following lemma combines
\cite[Theorem~5.7]{Marklof2003},
\cite[Lemma~6.6 and Proposition~6.5]{Marklof2003}, and
\cite[Corollaries~7.4 and~7.6]{Marklof2003}.

\begin{lemma}
\label{lem:marklof-nonescape}
Let $\Gamma_H<H$ be commensurable with
$\operatorname{SL}(2,\Z)\ltimes\Z^2$, let $\mu_Y$ be Haar probability
measure on $Y=\Gamma_H\backslash H$, and let $\alpha\notin\Q$.
For every bounded continuous $F:Y\to\C$ and every $h\in C_c(\R)$,
\begin{equation}\label{eq:bounded-equid}
 \lim_{v\to0}\int_\R F(\Gamma_Hg_\alpha n(u)a(v))h(u)\,du
 =\left(\int_YF\,d\mu_Y\right)\left(\int_\R h(u)\,du\right).
\end{equation}
Suppose in addition that $\alpha$ is Diophantine of type $\kappa$ in the
sense of \eqref{eq:diophantine}.  If $h_0\in C_c(\R)$ is nonnegative,
$0<\epsilon<1$, and $0<\epsilon'<1/(\kappa-1)$, then the function
\eqref{eq:cusp-majorant} satisfies
\begin{equation}\label{eq:majorant-bound}
 \limsup_{v\to0}\int_{|u|>v^{1-\epsilon}}
 F_R(g_\alpha n(u)a(v))h_0(u)\,du
 \ll_{\alpha,f_0,h_0,\epsilon,\epsilon'}R^{-\epsilon'/2}.
\end{equation}
Consequently, if $F$ is continuous and dominated by $F_R$, then
\begin{equation}\label{eq:unbounded-equid}
 \lim_{v\to0}\int_{|u|>v^{1-\epsilon}}
 F(g_\alpha n(u)a(v))h(u)\,du
 =\left(\int_YF\,d\mu_Y\right)\left(\int_\R h(u)\,du\right).
\end{equation}
The same statements hold on a finite cover of $Y$ and on a finite union of
rational translates of such covers.
\end{lemma}

\begin{lemma}
\label{lem:theta-domination}
For $f_1,f_2\in\mathcal S(\R^2)$, the function
$F=\Theta_{f_1}\overline{\Theta_{f_2}}$ satisfies
\eqref{eq:dominated} for a majorant of the form
\eqref{eq:cusp-majorant}.  This remains true for every component obtained by
a rational congruence condition and for every smooth angular cutoff.
\end{lemma}

\begin{proof}
Define
\[
 f_0(t)=\sup_{s\in\R}\sup_{\phi\in\R}
 |(f_1)_\phi(t,s)(f_2)_\phi(t,s)|.
\]
The metaplectic rotations of a Schwartz function are uniformly Schwartz for
$\phi$ in the compact circle.  Hence, for every $A>0$,
$f_0(t)\ll_A(1+|t|)^{-A}$.  Replace $f_0(t)$ by
$f_0(t)+f_0(-t)$ to make it even and nonnegative.

In the standard cusp $v>R$, expand both theta sums using
\eqref{eq:marklof-theta}.  Cauchy--Schwarz in the second lattice coordinate
and the definition of $f_0$ bound their product by
\[
 v\sum_{m\in\Z}f_0((m-y_1)v^{1/2})+O_A(v^{-A}).
\]
Move an arbitrary point of height at least $R$ to the standard cusp and sum
this estimate over $\Gamma_\infty\backslash\operatorname{SL}(2,\Z)$.  The
first term becomes \eqref{eq:cusp-majorant}, and the rapidly decreasing error
is absorbed in the constant $L$ in \eqref{eq:dominated}.  This is Marklof's
construction in \cite[Sections~7.2 and~8.4.3]{Marklof2003}.

A rational component is a finite linear combination of theta functions with
rational characteristics, so the same estimate holds after increasing its
constant.  Multiplication by a smooth angular cutoff replaces $f_j$ by a
Schwartz function and does not affect the argument.
\end{proof}

\subsection*{Rational theta data and the smoothed limit}

\begin{marklofdefinition}[Collision relation]\label{def:collision-relation}
Let $P_\alpha(m)=P_0(m)+\alpha\ell(m)+c$ on a lattice coset, where
$P_0$ is rational-valued, $\ell$ is a rational linear form, and
$\alpha\notin\Q$.  A pair $(m,n)$ is a \emph{bulk collision} if
\[
 \ell(m)=\ell(n),\qquad P_0(m)=P_0(n).
\]
The diagonal $m=n$ is always a bulk collision.  A \emph{collision graph} is
the graph of a map $\sigma$ for which $(m,\sigma(m))$ is a bulk collision for
all labels in its domain.
\end{marklofdefinition}

The terminology has a direct analytic meaning.  In the Haar integral of a
theta product, integration over the compact translation variable forces
$\ell(m)=\ell(n)$, and integration over the modular horocycle forces
$P_0(m)=P_0(n)$.

Let $A$ be a positive-definite symmetric $2\times2$ matrix with rational
entries, let $L\subset\Z^2$ be a finite-index sublattice, and fix
$c\in\Z^2$.  After interchanging the coordinates if necessary, write the
affine shift as $\boldsymbol\alpha=(\alpha,\beta)$, where
$\alpha\notin\Q$ and $\beta\in\Q$.  For $f\in\mathcal S(\R^2)$ and
$\tau=u+iv\in\mathbb H$, put
\begin{equation}\label{eq:rational-theta-component}
 \Theta_{A,L,c,\boldsymbol\alpha;f}(u,v)
 =v^{1/2}\sum_{m\in c+L}
 f\!\left(v^{1/2}A^{1/2}(m+\boldsymbol\alpha)\right)
 \e\!\left(\frac u2
 (m+\boldsymbol\alpha)^tA(m+\boldsymbol\alpha)\right).
\end{equation}
\begin{lemma}\label{lem:rational-theta}
The function in \eqref{eq:rational-theta-component} is a finite linear
combination of theta components on finite-volume arithmetic quotients of
$G^2$.  Along the associated affine horocycle, each component lies on one of
finitely many rational translates of $Y$, with irrational translation
coordinate $q\alpha+r$, where $q\in\Q^\times$ and $r\in\Q$.  Hence, if
$\alpha$ is Diophantine of finite type,
\cref{lem:marklof-nonescape,lem:theta-domination} apply componentwise.
\end{lemma}

\begin{proof}
Choose an integer clearing the denominators of $A$ and $\beta$.  Expanding
the congruence condition by
\[
 \one_{c+L}(m)=\frac1{[\Z^2:L]}
 \sum_{\chi\in\widehat{\Z^2/L}}\overline{\chi(c)}\chi(m)
\]
expresses \eqref{eq:rational-theta-component}, after a fixed rescaling of
$\tau$, as a finite linear combination of theta functions with rational
characteristics.  The standard finite Weil representation shows that their
products with conjugates descend to finite-volume arithmetic quotients of
$G^2$; see \cite{LionVergne1980,BerndtSchmidt1998}.

The rational coordinate $\beta$, the congruence characters, and the rational
characteristics contribute only finitely many rational translates of the
subgroup $H$ defined above.  On each translate, the remaining translation
coordinate is obtained from $\alpha$ by a rational affine change and is
therefore $q\alpha+r$ for some $q\in\Q^\times$ and $r\in\Q$.  Rational
affine changes preserve finite Diophantine type.  The
majorant argument in \cref{lem:theta-domination} is componentwise, while a
finite sum changes only its constants.  Thus the equidistribution and
nonescape statements apply to every component product.
\end{proof}

By a \emph{smooth angular weight} we mean a bounded function
$w\in C^\infty(\R^2)$ for which there exist $R_0>0$ and
$\omega\in C^\infty(S^1)$ such that
\[
 w(r\theta)=\omega(\theta)
 \qquad(\theta\in S^1,\ r\geq R_0).
\]
Thus $w$ is homogeneous of degree zero outside the ball $B(0,R_0)$.

\begin{markloflemma}\label{lem:marklof-smoothed}
Let $\{\lambda_m:m\in M\}$ be represented by the rational theta data of
\cref{lem:rational-theta}, and let $w$ be a smooth angular weight.  Suppose
that, for $w$ and $w^2$, the weighted Weyl law
\begin{equation}\label{eq:weighted-weyl}
 \lim_{X\to\infty}\frac1X\sum_{m\in M}
 w(m)\varphi(\lambda_m/X)
 =d(w)\int_0^\infty\varphi(t)\,dt
\end{equation}
holds for every $\varphi\in C_c^\infty(\R_{>0})$.  If the irrational affine
coordinate is Diophantine and $m=n$ is the only bulk collision on
$\operatorname{supp}(w)\times\operatorname{supp}(w)$, then
\begin{align}\label{eq:smoothed-offdiag}
 &\lim_{X\to\infty}\frac1X
 \sum_{\substack{m,n\in M\\m\ne n}}
 w(m)w(n)
 \psi_1\!\left(\frac{\lambda_m}{X}\right)
 \psi_2\!\left(\frac{\lambda_n}{X}\right)
 \widehat h(\lambda_m-\lambda_n)\notag\\
 &\qquad=d(w)^2
 \left(\int_\R\widehat h(s)\,ds\right)
 \left(\int_0^\infty\psi_1(t)\psi_2(t)\,dt\right)
\end{align}
for $\psi_1,\psi_2\in C_c^\infty(\R_{>0})$ and $h\in C_c(\R)$.
\end{markloflemma}

\begin{proof}
For the basic affine lattice, this is the weighted theta mean-value calculation
of \cite[Proposition~8.5]{GriffinMarklof2019}, followed by subtraction of the
diagonal; see also \cite[Section~8.6]{Marklof2003}.  By
\cref{lem:rational-theta}, the rational form and congruence coset give finitely
many theta products on finite covers and rational translates.  The Ratner and
nonescape arguments apply to each component with the same estimates by
\cref{lem:marklof-nonescape,lem:theta-domination}.  Summing the components
proves the formula.
\end{proof}

\begin{markloflemma}\label{lem:remove-cutoffs}
Let $\mathcal D\subset\R^2$ be a closed cone such that
$\mathcal D\cap S^1$ is a finite union of closed arcs.  Suppose
\cref{lem:marklof-smoothed} holds for smooth angular weights near
$\mathcal D$, and that every nonidentity collision graph meeting
$\mathcal D\times\mathcal D$ does so only on its boundary.  Put
\[
 M_{\mathcal D}=M\cap\mathcal D,\qquad
 N_{\mathcal D}(X)=\#\{m\in M_{\mathcal D}:\lambda_m\leq X\}.
\]
If $N_{\mathcal D}(X)=DX+o(X)$ with $D>0$, then, for every
$f\in C_c^\infty(\R)$,
\[
 \lim_{X\to\infty}\frac1{N_{\mathcal D}(X)}
 \sum_{\substack{m,n\in M_{\mathcal D}\;m\ne n\\
                  \lambda_m,\lambda_n\leq X}}
 f(\lambda_m-\lambda_n)
 =D\int_\R f(t)\,dt.
\]
\end{markloflemma}

\begin{proof}
Approximate the angular and radial indicators from above and below by smooth
weights.  The weighted Weyl law makes the boundary-strip and annular errors
$O(\eta)$; Cauchy--Schwarz and \cref{lem:marklof-smoothed} make every
corresponding pair term $O_f(\eta^{1/2})$.  Letting $\eta\to0$ proves the
claim.  This is the standard approximation argument of
\cite[Section~8.6 and Appendix~A.6]{Marklof2003}.
\end{proof}

\section{A sector version of the inhomogeneous quadratic-form theorem}
\label{sec:marklof}

We isolate the homogeneous-dynamics input in the exact form required by the
Sutherland spectrum.  This makes explicit the three points not literally
present in the most familiar statement of Marklof's theorem: the rational
coefficient $3$, the congruence coset, and the Weyl sector.

For $r\in\Z/3\Z$, put
\begin{align*}
 \Lambda_r&=\{(u,v)\in\Z^2:u\equiv v\pmod2,
                    \ v\equiv r\pmod3\},\\
 \cC&=\{(u,v)\in\R^2:u\geq|v|\},\qquad
 \cC^+=\{(u,v)\in\R^2:u\geq v\geq0\}.
\end{align*}
For $z=(u,v)$ define
\[
 \xi_g(z)=\frac1{12}\bigl(3(u+2g)^2+v^2\bigr).
\]
If $M\subset\Z^2$ and $\xi:M\to\R_{>0}$, write
$N_M(X)=\#\{z\in M:\xi(z)\leq X\}$.  We say that its values have
\emph{pair-correlation intensity $D$} if, for every $F\in C_c^\infty(\R)$,
\begin{equation}\label{eq:intensity-definition}
 \lim_{X\to\infty}\frac1{N_M(X)}
 \sum_{\substack{z,z'\in M\\z\ne z',\ \xi(z),\xi(z')\leq X}}
 F(\xi(z)-\xi(z'))=D\int_\R F(t)\,dt.
\end{equation}

\begin{proposition}\label{prop:sector-marklof}
Let $g$ be irrational and Diophantine of finite type.
For $r=1,2$, the values $\xi_g(z)$ with $z\in\Lambda_r\cap\cC$ have Poissonian
pair correlation of intensity $\pi/(3\sqrt3)$.  For $r=0$, the values with
$z\in\Lambda_0\cap\cC^+$ have Poissonian pair correlation of intensity
$\pi/(6\sqrt3)$.
\end{proposition}

\begin{proof}
Set
\[
 L=\{(u,v)\in\Z^2:u\equiv v\pmod2, v\equiv0\pmod3\},
 \qquad A=\begin{pmatrix}3&0\\0&1\end{pmatrix},
 \qquad \boldsymbol\alpha=(2g,0).
\]
For $c_r=(r,r)$, where $r$ is represented by $0,1$, or $2$, one has
$\Lambda_r=c_r+L$.  Indeed, subtraction of $c_r$ makes the second coordinate
divisible by three and preserves the parity relation, and the converse is
immediate.  Moreover,
\begin{equation}\label{eq:xi-rational-data}
 12\xi_g(z)=(z+\boldsymbol\alpha)^tA(z+\boldsymbol\alpha).
\end{equation}
Thus $A$, $L$, and $c_r$ satisfy the hypotheses of
\cref{lem:rational-theta}.  The factor $12$ only rescales $u$ in
\cref{eq:theta-pair-identity}; it neither changes automorphy nor the
nonescape estimate.

Let $\chi$ be a smooth angular function supported in the interior of $\cC$
or $\cC^+$, and let $\psi_1,\psi_2\in C_c^\infty(\R_{>0})$.  Apply
\cref{lem:theta-identity} to the sums
\[
 X^{-1/2}\sum_{z\in\Lambda_r}
 \chi(z+\boldsymbol\alpha)
 \psi_j(\xi_g(z)/X)\e(u\xi_g(z)/2).
\]
By \eqref{eq:xi-rational-data} and \cref{lem:rational-theta}, their product is
a finite sum of automorphic theta products on finite covers of the quotient
$Y$ defined above.  The irrational translation coordinate is $2g$.  To check
the Diophantine condition explicitly, if $a\in\Z$ and $b\geq1$, then
\[
 \left|2g-\frac ab\right|
 =2\left|g-\frac{a}{2b}\right|
 \geq2c(2b)^{-\kappa}.
\]
Hence $2g$ is Diophantine of the same finite type as $g$.
By \cref{lem:marklof-nonescape,lem:theta-domination}, the theta products are
therefore uniformly integrable and equidistribute on the explicit
finite-volume $H$-quotients.

By \cref{lem:collision}, the sector restriction removes every nontrivial
collision away from the fixed wall $p=q$, and that wall contains only
$O_g(X^{1/2})$ labels by \cref{lem:weyl}, so its contribution is absorbed by
the error term.

The bijection $(p,q)\mapsto(u,v)=(p+q,p-q)$ and \cref{lem:weyl} give
$D=\pi/(3\sqrt3)$ for $r=1,2$ and $D=\pi/(6\sqrt3)$ for $r=0$.

For a smooth angular cutoff $\chi$, the Lipschitz principle used in
\cref{lem:weyl} proves \eqref{eq:weighted-weyl}; denote the resulting
weighted density by $D_\chi=d(\chi)$.  All hypotheses of
\cref{lem:marklof-smoothed} have now been verified.  Its off-diagonal
conclusion says
\[
 \lim_{X\to\infty}\frac1{D_\chi X}
 \sum_{z\ne z'}
 \chi(z+\boldsymbol\alpha)\chi(z'+\boldsymbol\alpha)
 \psi_1(\xi_g(z)/X)\psi_2(\xi_g(z')/X)
 \widehat h(\xi_g(z)-\xi_g(z'))
 =D_\chi\int_\R\widehat h\int_0^\infty\psi_1\psi_2.
\]
Finally, \cref{lem:remove-cutoffs} replaces the smooth weights by the sharp
sector and replaces the radial weights by the conditions
$\xi_g(z),\xi_g(z')\leq X$.  Since $D_\chi\to D$ for the upper and lower
angular approximations and $N_M(X)/(DX)\to1$, the resulting formula is
exactly \eqref{eq:intensity-definition}.  The theta weight is naturally a
function of $z+\boldsymbol\alpha$, whereas the sharp cone condition is imposed
on $z$.  The symmetric difference between these two cone conditions lies in
a strip of width $O_g(1)$ along the boundary rays.  Below energy $X$ the
strip has area $O_g(X^{1/2})$, hence relative area $O_g(X^{-1/2})$; the
discrepancy estimate in \cref{lem:remove-cutoffs} shows that it contributes
zero to the normalized pair limit.  This proves the proposition.
\end{proof}

\begin{remark}\label{rem:why-dioph}
The Diophantine assumption is used only in \eqref{eq:majorant-bound}.
Equation \eqref{eq:bounded-equid} for bounded continuous observables requires
only irrationality.  Pair correlation is represented by the unbounded theta
product, and a Liouville sequence of exceptionally short affine-lattice
vectors can carry mass into the cusp.  This is why irrationality alone is not
enough for the present proof.
\end{remark}

\section{Proof of the main theorem}\label{sec:proof}

\begin{proof}[Proof of \cref{thm:main}]
By \cref{lem:relative-spectrum}, the map
$(p,q)\mapsto E_{p,q}=Q_g(p,q)/3$ gives every relative eigenvalue in the
fixed momentum sector.  The change of variables \eqref{eq:uv} is invertible
on its image, with inverse
\[
 p=\frac{u+v}{2},\qquad q=\frac{u-v}{2}.
\]
The parity and cone conditions in \eqref{eq:uv-lattice} are exactly the
conditions that these two numbers are nonnegative integers.  The momentum
congruence becomes $v\equiv K\pmod3$.  Hence the label set is in bijection
with $\Lambda_{K\bmod3}\cap\cC$, and the identity
$4Q_g=3(u+2g)^2+v^2$ shows that its energy is $\xi_g(u,v)$.

When $K\equiv0\pmod3$, the reflection $(p,q)\mapsto(q,p)$ becomes
$(u,v)\mapsto(u,-v)$.  Choosing the representative $p\geq q$ is therefore
equivalent to imposing $v\geq0$, namely to restricting to $\cC^+$.  By
\cref{lem:collision}, no two distinct retained labels have the same energy.
Thus the sets in \cref{prop:sector-marklof} are precisely the desymmetrized
spectra $\mathcal E_{K,g}$.

The Weyl law \eqref{eq:intro-weyl} is \cref{lem:weyl}.  Let
$f\in C_c^\infty(\R)$ and put $F(x)=f(\rho_Kx)$.  By the definition
\eqref{eq:intensity-definition} and \cref{prop:sector-marklof},
\[
 \lim_{X\to\infty}R_{K,g}(X,f)=
 \rho_K\int_\R F(x)\,dx
 =\rho_K\int_\R f(\rho_Kx)\,dx
 =\int_\R f(y)\,dy,
\]
where the last equality uses $y=\rho_Kx$.  This proves both assertions.
\end{proof}

\par\medskip\noindent\hrulefill\par

\subsection*{Acknowledgments}
The author thanks Jens Marklof for helpful comments on an earlier draft, in
particular for pointing out the relevance of
\cite[Proposition~8.5]{GriffinMarklof2019} to the weighted argument.

\vspace{0.25\baselineskip}

\subsection*{AI use declaration}
Generative artificial intelligence, ChatGPT (OpenAI), was used in developing
and checking preliminary proof strategies, verifying calculations and logical
consistency, revising the exposition, checking bibliographic information, and
assisting with LaTeX preparation. All mathematical arguments, statements, and
citations in the final manuscript were independently evaluated and verified by
the author. The author made every final decision concerning the content and
assumes full responsibility for the correctness and originality of the paper.

\bibliographystyle{alpha}
\bibliography{biblio}

\newcommand{\etalchar}[1]{$^{#1}$}
\begin{thebibliography}{MnFM{\etalchar{+}}06}

\bibitem[BF97]{BakerForrester1997}
T.~H. Baker and P.~J. Forrester.
\newblock The {Calogero--Sutherland} model and generalized classical
  polynomials.
\newblock {\em Communications in Mathematical Physics}, 188(1):175--216, 1997.

\bibitem[BS98]{BerndtSchmidt1998}
R.~Berndt and R.~Schmidt.
\newblock {\em Elements of the Representation Theory of the {Jacobi} Group},
  volume 163 of {\em Progress in Mathematics}.
\newblock Birkh\"auser, Basel, 1998.

\bibitem[BT77]{BerryTabor1977}
M.~V. Berry and M.~Tabor.
\newblock Level clustering in the regular spectrum.
\newblock {\em Proceedings of the Royal Society of London. Series A},
  356:375--394, 1977.

\bibitem[EMM05]{EskinMargulisMozes2005}
A.~Eskin, G.~Margulis, and S.~Mozes.
\newblock Quadratic forms of signature $(2,2)$ and eigenvalue spacings on
  rectangular $2$-tori.
\newblock {\em Annals of Mathematics}, 161(2):679--725, 2005.

\bibitem[GM19]{GriffinMarklof2019}
J.~Griffin and J.~Marklof.
\newblock Quantum transport in a low-density periodic potential: homogenisation
  via homogeneous flows.
\newblock {\em Pure and Applied Analysis}, 1(4):571--614, 2019.

\bibitem[Hal24]{Hallnas2024}
M.~Halln\"as.
\newblock Calogero--moser--sutherland systems, 2024.

\bibitem[KMW26]{KimMarklofWelsh2026}
W.~Kim, J.~Marklof, and M.~Welsh.
\newblock Values of ternary quadratic forms at integers and the {Berry--Tabor}
  conjecture for $3$-tori, 2026.
\newblock \href{https://arxiv.org/abs/2601.03209}{arXiv:2601.03209 [math.NT]}.

\bibitem[LMW23]{LindenstraussMohammadiWang2023}
E.~Lindenstrauss, A.~Mohammadi, and Z.~Wang.
\newblock Quantitative equidistribution and the local statistics of the
  spectrum of a flat torus.
\newblock {\em Journal d'Analyse Math\'ematique}, 151(1):181--234, 2023.

\bibitem[LV80]{LionVergne1980}
G.~Lion and M.~Vergne.
\newblock {\em The {Weil} Representation, {Maslov} Index and Theta Series},
  volume~6 of {\em Progress in Mathematics}.
\newblock Birkh\"auser, Boston, 1980.

\bibitem[Man07]{Mangazeev2007}
V.~V. Mangazeev.
\newblock An analytic formula for the $a_2$ {Jack} polynomials.
\newblock {\em SIGMA}, 3:014, 2007.

\bibitem[Mar02]{Marklof2002}
J.~Marklof.
\newblock Pair correlation densities of inhomogeneous quadratic forms. {II}.
\newblock {\em Duke Mathematical Journal}, 115(3):409--434, 2002.

\bibitem[Mar03a]{MarklofCorrection2003}
J.~Marklof.
\newblock Correction to ``pair correlation densities of inhomogeneous quadratic
  forms. {II}''.
\newblock {\em Duke Mathematical Journal}, 120(1):227--228, 2003.

\bibitem[Mar03b]{Marklof2003}
J.~Marklof.
\newblock Pair correlation densities of inhomogeneous quadratic forms.
\newblock {\em Annals of Mathematics}, 158(2):419--471, 2003.

\bibitem[MnFM{\etalchar{+}}06]{MunozEtAl2006}
L.~Mu\~noz, E.~Faleiro, R.~A. Molina, A.~Rela\~no, and J.~Retamosa.
\newblock Spectral statistics in noninteracting many-particle systems.
\newblock {\em Physical Review E}, 73(3):036202, 2006.

\bibitem[OP83]{OlshanetskyPerelomov1983}
M.~A. Olshanetsky and A.~M. Perelomov.
\newblock Quantum integrable systems related to {Lie} algebras.
\newblock {\em Physics Reports}, 94(6):313--404, 1983.

\bibitem[PZB{\etalchar{+}}93]{PoilblancEtAl1993}
D.~Poilblanc, T.~Ziman, J.~Bellissard, F.~Mila, and G.~Montambaux.
\newblock Poisson versus {GOE} statistics in integrable and non-integrable
  quantum hamiltonians.
\newblock {\em Europhysics Letters}, 22(7):537--542, 1993.

\bibitem[Sar97]{Sarnak1997}
P.~Sarnak.
\newblock Values at integers of binary quadratic forms.
\newblock In {\em Harmonic Analysis and Number Theory (Montreal, 1996)},
  volume~21 of {\em CMS Conference Proceedings}, pages 181--203. American
  Mathematical Society, Providence, RI, 1997.

\bibitem[SIB21]{SekkouriIzrailevBorgonovi2021}
S.~M. Sekkouri, F.~Izrailev, and F.~Borgonovi.
\newblock Spectrum statistics in the integrable {Lieb--Liniger} model.
\newblock {\em Physical Review E}, 104(3):034212, 2021.

\bibitem[Sut71]{Sutherland1971}
B.~Sutherland.
\newblock Exact results for a quantum many-body problem in one dimension.
\newblock {\em Physical Review A}, 4(5):2019--2021, 1971.

\bibitem[Sut72]{Sutherland1972}
B.~Sutherland.
\newblock Exact results for a quantum many-body problem in one dimension. {II}.
\newblock {\em Physical Review A}, 5(3):1372--1376, 1972.

\bibitem[SV20]{StrombergssonVishe2020}
A.~Str\"ombergsson and P.~Vishe.
\newblock An effective equidistribution result for
  {${\operatorname{SL}(2,\mathbb{R})\ltimes(\mathbb{R}^2)^{\oplus k}}$} and
  application to inhomogeneous quadratic forms.
\newblock {\em Journal of the London Mathematical Society}, 102(1):143--204,
  2020.

\bibitem[UKR23]{UrbinaKellyRichter2023}
J.~D. Urbina, M.~Kelly, and K.~Richter.
\newblock Periodic orbit theory of {Bethe}-integrable quantum systems: an
  $n$-particle {Berry--Tabor} trace formula.
\newblock {\em Journal of Physics A: Mathematical and Theoretical},
  56(21):214001, 2023.

\end{thebibliography}

\end{document}